\documentclass{llncs}
\usepackage{epsfig}

\usepackage{makeidx}  
\usepackage{amsfonts}
\usepackage{amssymb}
\usepackage{amsmath}

\newcommand{\R}{{I\!\!R}}

\newcommand{\C}{{\mathbb{C}}}

\def\R{{\rm I}\! {\rm R}}

\def\X{{\bf X}}

\newtheorem{algorithm}{Algorithm}[section]
\newtheorem{assumption}{Assumption}[section]

\begin{document}

\pagestyle{headings}

\title{Multiple Iterative Splitting method for Higher order and Integro-differental equations}
\author{J\"urgen Geiser and Thomas Zacher}
\institute{Ernst-Moritz-Arntz University of Greifswald, \\
Department of Physics, \\
Domstr. 14, D-17487  Greifswald, Germany \\
\email{juergen.geiser@uni-greifswald.de}}
\maketitle

\begin{abstract}

In this paper we present an extension of standard iterative
splitting schemes to multiple splitting schemes for solving
higher order differential equations.

We are motivated by dynamical systems,
which occur in dynamics of the electrons in the
plasma using a simplified Boltzmann equation.
Oscillation problems in spectroscopy problems using
wave-equations.

The motivation arose to simulate active plasma resonance spectroscopy
which is used for plasma diagnostic techniques, see \cite{braith2009}, \cite{lapke2010}
and \cite{oberrath2011}.

\end{abstract}



\section{Introduction}

We motivate our studying on simulating a active
plasma resonance spectroscopy which is well established in
plasma diagnostic techniques.

To study the model with simulation models, we concentrate on
an abstract kinetic model, which described the
dynamics of electrons in the plasma by using a 
Boltzmann equation. The Boltzmann equation is coupled
with the electric field and we obtain coupled 
partial differential equations.

The paper is outlined as follows.

In section \ref{modell} we present our mathematical model 
and a possible reduced model for the further approximations.

The functional analytical setting with the higher 
order differential equations are discussed in section \ref{higher}.

The splitting schemes are presented in in Section \ref{splitt}.

Numerical experiments are done in Section \ref{num}.
In the contents, that are given in Section \ref{concl}, 
we summarize our results.

\section{Mathematical Model}
\label{modell}

In the following a model is presented due to the 
motivation in  \cite{braith2009}, \cite{lapke2010}
and \cite{oberrath2011}.

The models consider a fluid dynamical approach
of the natural ability of plasmas to resonate
in the near of the electron plasma frequency $\omega_{pe}$.

Here we specialize to an abstract kinetic model to
describe the dynamics of the electrons in the plasma, that
allows to do the resonation.

The Boltzmann equation for the electron particles  are given as
\begin{eqnarray}
\label{neutron}
&& \frac{\partial f(x,v,t)}{\partial t} = - v \cdot \nabla_x f(x,v,t) - \frac{e}{m_e} \nabla_x \phi \cdot \nabla_v f(x,v,t) \nonumber \\
&&  - \sigma(x,v,t)  f(x,v,t) + \int_V \kappa(x,v,v') f(x, v', t) \; dv' , \\
&& f(x,v,0) = f_0(x,v) , 
\end{eqnarray}
and boundary conditions are postulated at the boundaries of $P$ (plasma).

In front of the materials we assume complete reflection of the
electrons due to the sheath $f(v_{||} + v_{\perp})$ with $v_{||}$
is the parallel and $v_{\perp}$ perpendicular to the surface normal
vector. $\phi$ is the electric field.

\section{Higher order differential equations}
\label{higher}

We consider the abstract homogeneous Cauchy problem in a Banach space $\X \in \R ^n$:
\begin{eqnarray}
\label{diff_1}
&& A_0 \frac{d^n {\bf u}(t)}{d^n t} + A_1 \frac{d^{n-1} {\bf u}(t)}{dt^{n-1}} {\bf u}(t) + \ldots +  A_{n} =  {\bf f}(t) , \\
&& \frac{d^{i-1} {\bf u}(t)}{dt^{n-1}} = {\bf u}_{i-1}, i = 1 , \ldots, n ,
\end{eqnarray}
where $A_0 , \ldots, A_n \in \X \times \X$ are bounded operators and $|| \cdot ||$ is the
corresponding norm in $\X$ and let $|| \cdot ||_{L(\X)}$ be the induced operator norm.

For the transformation we have the following assumptions:
\begin{assumption}
\label{assum1}

1.) The function $f(t)$ is given as:
\begin{eqnarray}
\label{acp_1}
&&  {\bf f}(t) = 0, 
\end{eqnarray}
otherwise we solve a non-autonomous equation.

2.) We assume that the characteristic polynomial:
\begin{eqnarray}
\label{assum_2}
&&  A_0 \lambda^n + A_1 \lambda^{n-1} + \ldots + A_{n} = 0 ,
\end{eqnarray}
has solution of complex valued matrices in $\X \times \X \in \C^m \times \C^m$,
given as:
\begin{eqnarray}
\label{assum_2}
 (\lambda I - B_1 ) (\lambda I - B_2) + \ldots + ( \lambda I - B_{n}) = 0 ,
\end{eqnarray}

\end{assumption}

\begin{corollary}

The higher order differential equation (\ref{diff_1}) can be decoupled with the
assumptions \ref{assum1} to the following differential equation:
\begin{eqnarray}
\label{diff_2}
&& \frac{d {\bf u}_1(t)}{d t} - B_1 {\bf u}_1 = 0 , \\
&& \frac{d {\bf u}_2(t)}{d t} - B_2 {\bf u}_2 = 0 , \\
&& \ldots \\
&& \frac{d {\bf u}_n(t)}{d t} - B_n {\bf u}_n= 0 ,
\end{eqnarray}
where the analytical solution is given as:
\begin{eqnarray}
\label{ana_1}
&& {\bf u}(t) = \sum_{i=1}^n \; d_i {\bf u}_{i, 0} = \sum_{i=1}^n  \exp(B_i t ) \; d_i {\bf u}_{i, 0}.
\end{eqnarray}
and $d_i$ are given via the initial conditions.

\end{corollary}

\begin{proof}

The solutions can be derived via the characteristics polynomial (idea of scalar linear differential equations) and the idea of the superposition of the 
linear combined solutions.
\end{proof}

\begin{remark}
The initial conditions are computed by solving the 
Vandermode matrix, see  the ideas in \cite{luther2004}. \\
We have to solve: 
\begin{eqnarray}
\label{diff_21}
\left(
\begin{array}{c c c c}
I & I & \ldots & I \\
B_1 & B_2 & \ldots & B_n \\
B_1^2 & B_2^2 & \ldots & B_n^2 \\
\vdots & \vdots & \ddots & \vdots \\
B_1^{n-1} & B_2^{n-1} & \ldots & B_n^{n-1} 
\end{array}
\right) \cdot
\left(
\begin{array}{c}
d_1 u_{1,0}  \\
d_2 u_{2,0}  \\
d_3 u_{3,0}  \\
\vdots  \\
d_n u_{n,0}
\end{array}
\right) =
\left(
\begin{array}{c}
u(0)  \\
\frac{\partial}{\partial t} u(0) \\
\frac{\partial^2}{\partial t^2} u(0)  \\
\vdots  \\
\frac{\partial^{n-1}}{\partial t^{n-1}} u(0)
\end{array}
\right) 
\end{eqnarray}

\end{remark}

A further simplification can be done to rewrite the integral-differential 
equation in two first order differential equations.
Later such a reduction allows us to apply fast iterative splitting 
methods.

\begin{corollary}

The higher order differential equation (\ref{diff_1}) can be transformed with the
assumptions \ref{assum1} to two first order differential equation:
\begin{eqnarray}
\label{diff_2_1}
&& \frac{d {\bf u}_1 (t)}{d t} = B_1 {\bf u}_1(t) \\
&& {\bf u}_1(0) = d_1 {\bf u}_{1,0}, \\
&&\ldots \\
\label{diff_2_2}
&& \frac{d {\bf u}_n (t)}{d t} = B_n {\bf u}_n(t) \\
&& {\bf u}_n(0) = d_n {\bf u}_{n,0},
\end{eqnarray}
where we have $B_i = B_{i1} + B_{i2}$ for $i = 1 \ldots, n$.

The analytical solution are given as:
\begin{eqnarray}
\label{ana_1}
&& {\bf u}(t) = \sum_{i=1}^n  \exp(B_i t ) d_i {\bf u}_{i,0} .\nonumber
\end{eqnarray}

\end{corollary}

\begin{proof}

The analytical solution of the first order differential equation (\ref{diff_2_1}) and (\ref{diff_2_2})
are given by each characteristic polynomial:
\begin{eqnarray}
&& \lambda_1 I - (B_{11} + B_{12}) = 0 , \\
&& \ldots \\
&& \lambda_n I -  (B_{n1} + B_{n2}) = 0 ,
\end{eqnarray}
while the solution is given as with the notations:
\begin{eqnarray}
&& \lambda_{i} I =  B_{i1} + B_{i2} , \; i = 1, \ldots, n ,
\end{eqnarray}
and therefore the analytical solution is given as (\ref{ana_1}).

Therefore this is the solution of our integro-differential equation (\ref{diff_1}) with the assumptions \ref{assum1}.

\end{proof}

\section{Splitting schemes}
\label{splitt}

The operator-splitting methods are used to solve complex models in
the geophysical and environmental physics, they are developed and
applied in \cite{stra68}, \cite{verwer98} and \cite{zla95}. This
ideas based in this article are solving simpler equations with
respect to receive higher order discretization methods for the
remain equations. For this aim we use the operator-splitting
method and decouple the equation as follows described.

In the following we concentrate on the iterative-splitting method.

\subsection{Iterative splitting method for Integro-differential equations}

The following algorithm is based on the iteration with fixed
splitting discretization step-size $\tau$, namely, on the time
interval $[t^n,t^{n+1}]$ we solve the following sub-problems
consecutively  for $i=0,2, \dots 2m$. (Cf. \cite{kan03} and
\cite{glow04}.)

\begin{eqnarray}
 && \frac{\partial c_{i j}(t)}{\partial t} = B_{i 1}
c_{i j}(t) \; + \; B_{i 2} c_{i j-1}(t) , \;
\mbox{with} \; \; c_{i j}(t^n) = d_i c^{n} \label{gleich_kap33a} \\
&& \mbox{and} \; c_{i 0}(t^n) = c^n \; , \; c_{i, -1} = 0.0 , \nonumber
\\\label{gleich_kap33b}
&& \frac{\partial c_{i, j+1}(t)}{\partial t} = B_{i 1} c_{i j}(t) \; + \; B_{i 2} c_{i, j+1}(t) , \; \\
&& \mbox{with} \; \; c_{i, j+1}(t^n) = d_i c^{n}\; , \nonumber
\end{eqnarray}
 where $c^n$ is the known split
approximation at the  time level  $t=t^{n}$. The split
approximation at the time-level $t=t^{n+1}$ is defined as
$c^{n+1}= \sum_{k= 1}^n c_{k, 2m+1}(t^{n+1})$.

\smallskip
\begin{theorem}
Let us consider the abstract Cauchy problem in a Banach space $\X$
\begin{equation}
\begin{array}{c}
{\displaystyle \partial_t c(t) = B_{i1} c(t) + B_{i2} c(t), \quad 0 < t
\leq T } \\
\noalign{\vskip 1ex} {\displaystyle c(0)=d_i c_{i, 0} }
\end{array} \label{eq:ACP}
i=1, \ldots, n ,
\end{equation}

\noindent where  $B_{i1}, B_{i2},B_{i1} + B_{i2}: \!  {\X} \rightarrow {\X} $ are given
linear  operators being generators of the $C_0$-semi-group and $c_{i,0}
\in \X$ is a given element. Then the iteration process
(\ref{gleich_kap33a})--(\ref{gleich_kap33b}) is convergent  and
the and the rate of the convergence is of second order.
\end{theorem}

\begin{proof}

The proof is done in the work of Geiser \cite{gei_2011}.

\end{proof}

The algorithm is given as:

\begin{algorithm}

\begin{eqnarray}
\label{gleich_kap3_3a} && \frac{\partial c^i(t)}{\partial t} = A
c^i(t) \; + \; B c^{i-1}(t), \;
\mbox{with} \; \; c^i(t^n) = c^{i-1}({t^{n+1}}) \\
&& \mbox{and the starting values} \; c^{0}(t^n) = c(t^n) \; \mbox{results of
last iteration} \; , \;
c^{-1}(t^n) = 0.0 , \nonumber \\
&& \frac{\partial c^{i+1}(t)}{\partial t} = A c^i(t) \; + \; B c^{i+1}(t), \; \\
&& \mbox{with} \; \; c^{i+1}(t^n) = c^{i}({t^{n+1}})\; , \nonumber \\
&& \epsilon > | c^{i+1} (t^{n+1}) - c^{i-1}(t^{n+1})| \mbox{Stop criterion} \\
&& \mbox{result for the next time-step} \\
&&  c(t^{n+1}) = c^{m}(t^{n+1}) , \; \mbox{for} \;  m \; \mbox{fulfill the stop-criterion}
\end{eqnarray}
for each $i=0,2, \dots$, where $c^n$ is the known split
approximation at the previous time level.

\end{algorithm}

In the following we concentrate on the iterative-splitting method.

\subsection{Iterative splitting method for higher order differential equations}

The following algorithm is based on the iteration with fixed
splitting discretization step-size $\tau$, namely, on the time
interval $[t^n,t^{n+1}]$ we solve the following sub-problems
consecutively  for $j=0,2, \dots 2m$. (Cf. \cite{kan03} and
\cite{glow04}.)

\begin{eqnarray}
 && \frac{\partial c_{i j}(t)}{\partial t} = B_{i 1}
c_{i j}(t) \; + \; B_{i 2} c_{i j-1}(t) , \;
\mbox{with} \; \; c_{i j}(t^n) = d_i c^{n} \label{gleich_kap33a} \\
&& \mbox{and} \; c_{i 0}(t^n) = c^n \; , \; c_{i, -1} = 0.0 , \nonumber
\\\label{gleich_kap33b}
&& \frac{\partial c_{i, j+1}(t)}{\partial t} = B_{i 1} c_{i j}(t) \; + \; B_{i 2} c_{i, j+1}(t) , \; \\
&& \mbox{with} \; \; c_{i, j+1}(t^n) = d_i c^{n}\; , \nonumber
\end{eqnarray}
 where $i = 1, \ldots, I$ are the number of equations.
Further $c^n$ is the known split
approximation at the  time level  $t=t^{n}$. The split
approximation at the time-level $t=t^{n+1}$ is defined as
$c^{n+1}= \sum_{k= 1}^n c_{k, 2m+1}(t^{n+1})$.

\smallskip
\begin{theorem}
Let us consider the abstract Cauchy problem in a Banach space $\X \subset \C$:
\begin{equation}
\begin{array}{c}
{\displaystyle \partial_t c(t) = B_{i1} c(t) + B_{i2} c(t), \quad 0 < t
\leq T } \\
\noalign{\vskip 1ex} {\displaystyle c(0)=d_i c_0 }
\end{array} \label{eq:ACP}
i=1, \ldots, n ,
\end{equation}

\noindent where $d_i \in \C$ is the constant based on the initial conditions, 
further $B_{i1}, B_{i2},B_{i1} + B_{i2}: \!  {\X} \rightarrow {\X} $ are given
linear  operators being generators of the $C_0$-semi-group and $c_0
\in \X$ is a given element. Then the iteration process
(\ref{gleich_kap33a})--(\ref{gleich_kap33b}) is convergent  and
the and the rate of the convergence is of second order.
\end{theorem}

\begin{proof}

The proof is done in the work of Geiser \cite{gei_2011}.

\end{proof}

The algorithm is given as:

\begin{algorithm}

\begin{eqnarray}
\label{gleich_kap3_3a} && \frac{\partial c_{i,j}(t)}{\partial t} = B_{1,i}
c_{i,j}(t) \; + \; B_{2,i} c_{i-1,j}(t), \;
\mbox{with} \; \; c_i(t^n) = c^{i-1}({t^{n+1}}) \\
&& \mbox{and the starting values} \; c_{i,0}(t^n) = c_i(t^n) \; \mbox{results of
last iteration} \; , \;
c_{i,-1}(t^n) = 0.0 , \nonumber \\
&& \frac{\partial c_{i,j+1}(t)}{\partial t} = B_{1,i} c_{i,j}(t) \; + \; B_{2,i} c_{i,j+1}(t), \; \\
&& \mbox{with} \; \; c^{i+1}(t^n) = c^{i}({t^{n+1}})\; , \nonumber \\
&& \epsilon > | c^{i+1} (t^{n+1}) - c^{i-1}(t^{n+1})| \mbox{Stop criterion} \\
&& \mbox{result for the next time-step} \\
&&  c(t^{n+1}) = c^{m}(t^{n+1}) , \; \mbox{for} \;  m \; \mbox{fulfill the stop-criterion}
\end{eqnarray}
for each $j=0,2, \dots$, where $c^n$ is the known split
approximation at the previous time level.

Further $B_i = B_{1,i} + B_{2,i}$ is a decomposition of the matrix $B_i$.
\end{algorithm}

We reformulate to an algorithm that deals only with real numbers
and rewrite:
\begin{eqnarray}
&& \partial_t ( c_{re}(t) + i c_{im}(t) ) = ( B_{re,i 1} + i B_{im, i 1} ) (c_{re}(t) + i c_{im}(t) \\
&& + ( B_{re,i 2} + i B_{im, i 2} ) (c_{re}(t) + i c_{im}(t)), \quad 0 < t
\leq T  , \nonumber \\
&& (c_{re}(0) + i c_{im}(0))=(d_{i, re} + i d_{i, im}) (c_{re, 0} + i c_{im, 0})  \nonumber
i=1, \ldots, n ,
\end{eqnarray}

We have the following algorithm:
\begin{algorithm}

In the following, we have two iteration processes:
\begin{itemize}
\item First iteration process: $j = 1, \ldots, J$ iterates over the
decomposition of the matrices $B_1, B_2$.
\item Second iteration process: $k = 1, \ldots, K$ iterates over the
real and imaginary parts.
\end{itemize}

We start with $j=0, k=0$, 

First we iterate over $j$
\begin{eqnarray*}
 && \frac{\partial (c_{i,re}^{j,k})(t)}{\partial t} = B_{1,re} c_{i,re}^{j,k}(t) \; + \; B_{2,re} c_{i,re}^{j-1,k}(t) \; \nonumber \\
&& - \;  B_{1,im} c_{i,im}^{j-1,k-1}(t)\; - \; B_{2,im} c_{i,im}^{j-1,k-1}(t) , \;  \nonumber \\
&& \frac{\partial (c_{i,im}^{j,k})(t)}{\partial t} = B_{1,re} c_{i,im}^{j,k}(t) \; + \; B_{2,re} c_{i,im}^{j-1,k}(t) \; \nonumber \\
&& + \;  B_{1,im} c_{i,re}^{j-1,k-1}(t)\; + \; B_{2,im} c_{i,re}^{j-1,k-1}(t) , \;  \nonumber \\
 && \mbox{with the initial condition} \; \; c_{i,re}^{j,k}(t^n) = c_{i,re}^{j,k,n} ,  c_{i,im}^{j,k}(t^n) = c_{i,im}^{j,k,n}  \\
 && \mbox{with the starting condition} \; \;  c_{i,re}^{-1,k}(t^n) = 0 ,  c_{i,im}^{-1,k}(t^n) = 0 \\
\\
 && \frac{\partial (c_{i,re}^{j+1,k})(t)}{\partial t} = B_{1,re} c_{i,re}^{j,k}(t) \; + \; B_{2,re} c_{i,re}^{j+1,k}(t) \; \nonumber \\
&& - \;  B_{1,im} c_{i,im}^{j-1,k-1}(t)\; - \; B_{2,im} c_{i,im}^{j-1,k-1}(t) , \;  \nonumber \\
&& \frac{\partial (c_{i,im}^{j+1,k})(t)}{\partial t} = B_{1,re} c_{i,im}^{j,k}(t) \; + \; B_{2,re} c_{i,im}^{j+1,k}(t) \; \nonumber \\
&& + \;  B_{1,im} c_{i,re}^{j-1,k-1}(t)\; + \; B_{2,im} c_{i,re}^{j-1,k}(t) , \;  \nonumber \\
 && \mbox{with the initial condition} \; \; c_{i,re}^{j+1,k}(t^n) = c_{i,re}^{j+1,k,n} ,  c_{i,im}^{j+1,k}(t^n) = c_{i,im}^{j+1,k,n}  \\
\end{eqnarray*}

if $j = J$ or the iteration error over $j$ is less $err$
then we iterate over $k$.

Further $c^n$ is the known split
approximation at the  time level  $t=t^{n}$, cf.
\cite{fargei05}.

Further $B_i = B_{1,i} + B_{2,i}$ is a decomposition of the matrix $B_i$.
\end{algorithm}

\section{Experiments for the Plasma resonance spectroscopy}
\label{num}

In the following, we present different examples.

\subsection{First Example: Matrix problem with integral term}

We deal with a simpler integro-differential equations:

\begin{eqnarray}
\label{equation1}
  \frac{dc}{dt} - A c(t) dt = B \int c(t') \; dt', \; t \in [0,1],
\end{eqnarray}
and the transformed second order differential equation is given as:
\begin{eqnarray}
\label{eq20}
 \partial_{tt} c  && =  A  \partial_t c + B c \\
\end{eqnarray}
and the operators for the splitting scheme are given as:
\begin{eqnarray}
\label{eq20}
 \tilde{A} = -\frac{A}{2} , \; \tilde{B} = \sqrt{\frac{A \; A^T}{4} - B}
\end{eqnarray}
while $\tilde{A}^T$ is the transposed matrix of $\tilde{A}$.

The matrices are given as
\begin{equation}
\label{num_9}
A = \left(
\begin{array}{c c c c c c c c c c}
- 0.01 & 0.01 & 0 &  \ldots\\
  0.01 & - 0.01 & 0 & \ldots \\
  0.01 & 0.01 & -0.02 & 0 & \ldots  \\
  0.01 & 0.01 & 0.01 & - 0.03 &  0 & \ldots  \\
  \vdots \\
  0.01 & 0.01 & 0.01 & 0.01 & 0.01 & 0.01 & 0.01 & 0.01 &  - 0.08 & 0  \\
  0.01 & 0.01 & 0.01 & 0.01 &  0.01 & 0.01 & 0.01 & 0.01 &  0 & -0.08  \\
\end{array} \right) \; ,
\end{equation}
\begin{equation}
\label{num_9_1}
B = \left(
\begin{array}{c c c c c c c c c c}
  - 0.08 & 0 & 0.01 & 0.01 &  0.01 & 0.01 & 0.01 & 0.01 &  0.01 & 0.01  \\
   0 & -0.08 & 0.01 & 0.01 &  0.01 & 0.01 & 0.01 & 0.01 &  0.01 & 0.01  \\
 \vdots \\
  0 & 0 & 0 & 0 &  0 & 0 & 0 & - 0.02 &  0.01 & 0.01  \\
0 &  0 & 0 & 0 & 0 &  0 & 0 & 0 &  -0.01 &  0.01  \\
0 &  0 & 0 & 0 & 0 &  0 & 0 & 0 &  0.01 &  -0.01  \\
\end{array} \right) \; .
\end{equation}

The Figure \ref{first_1} present the numerical errors between the exact and the
numerical solution. Here we obtain results for one-side and two-side iterative schemes on operators $A$ and $B$. 
\begin{figure}[ht]
\begin{center}  
\includegraphics[width=9.0cm,angle=-0]{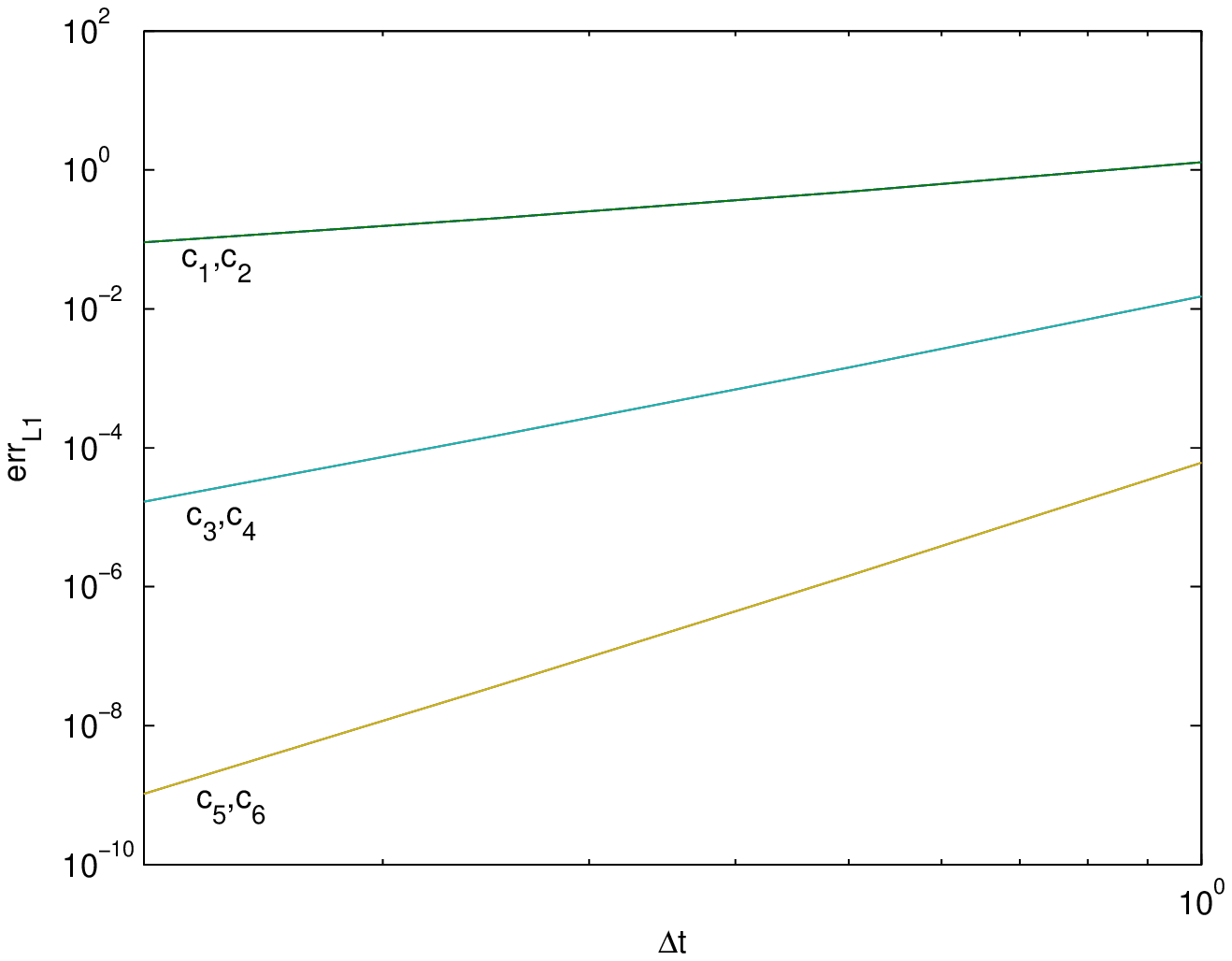} 
\includegraphics[width=9.0cm,angle=-0]{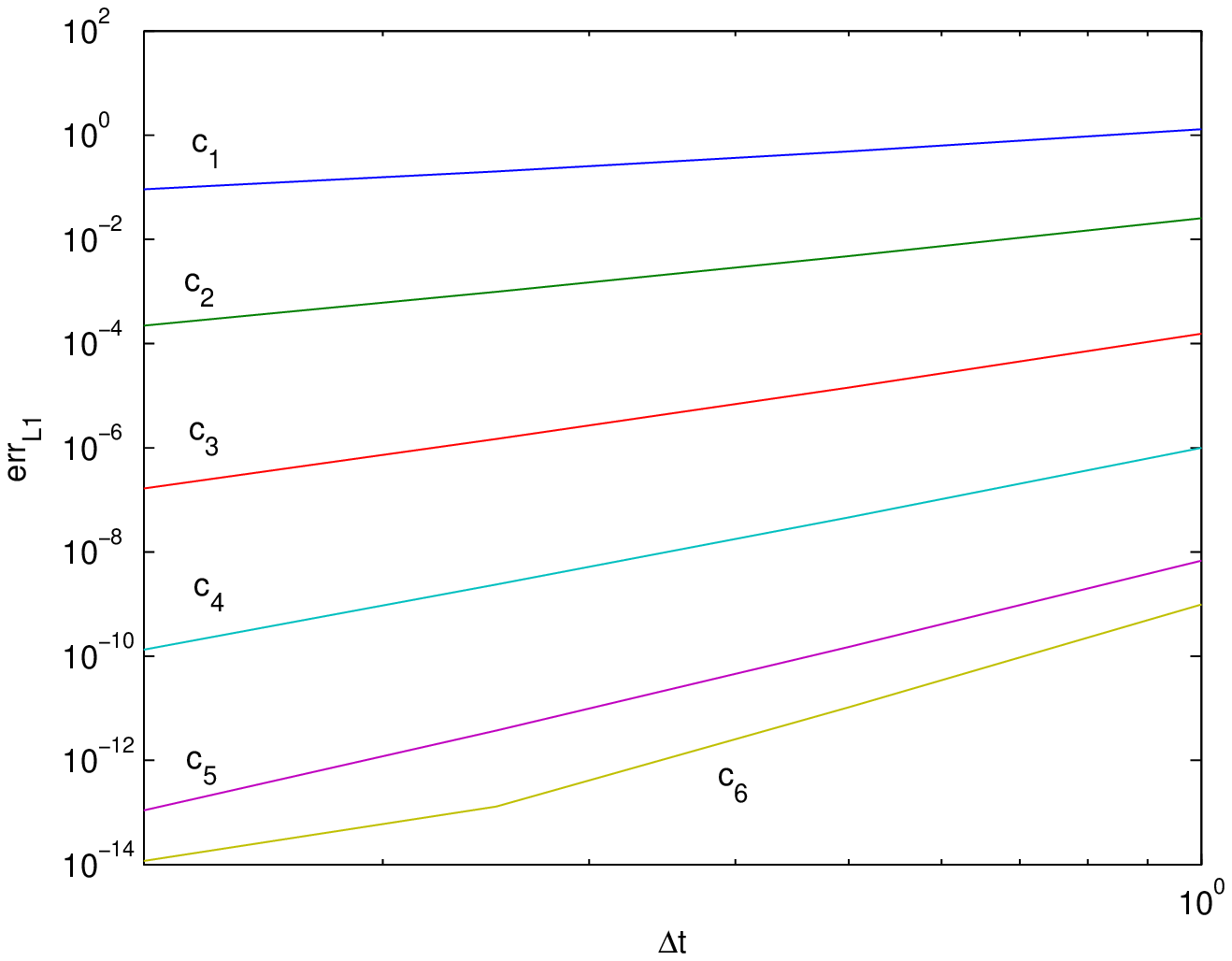} 
\includegraphics[width=9.0cm,angle=-0]{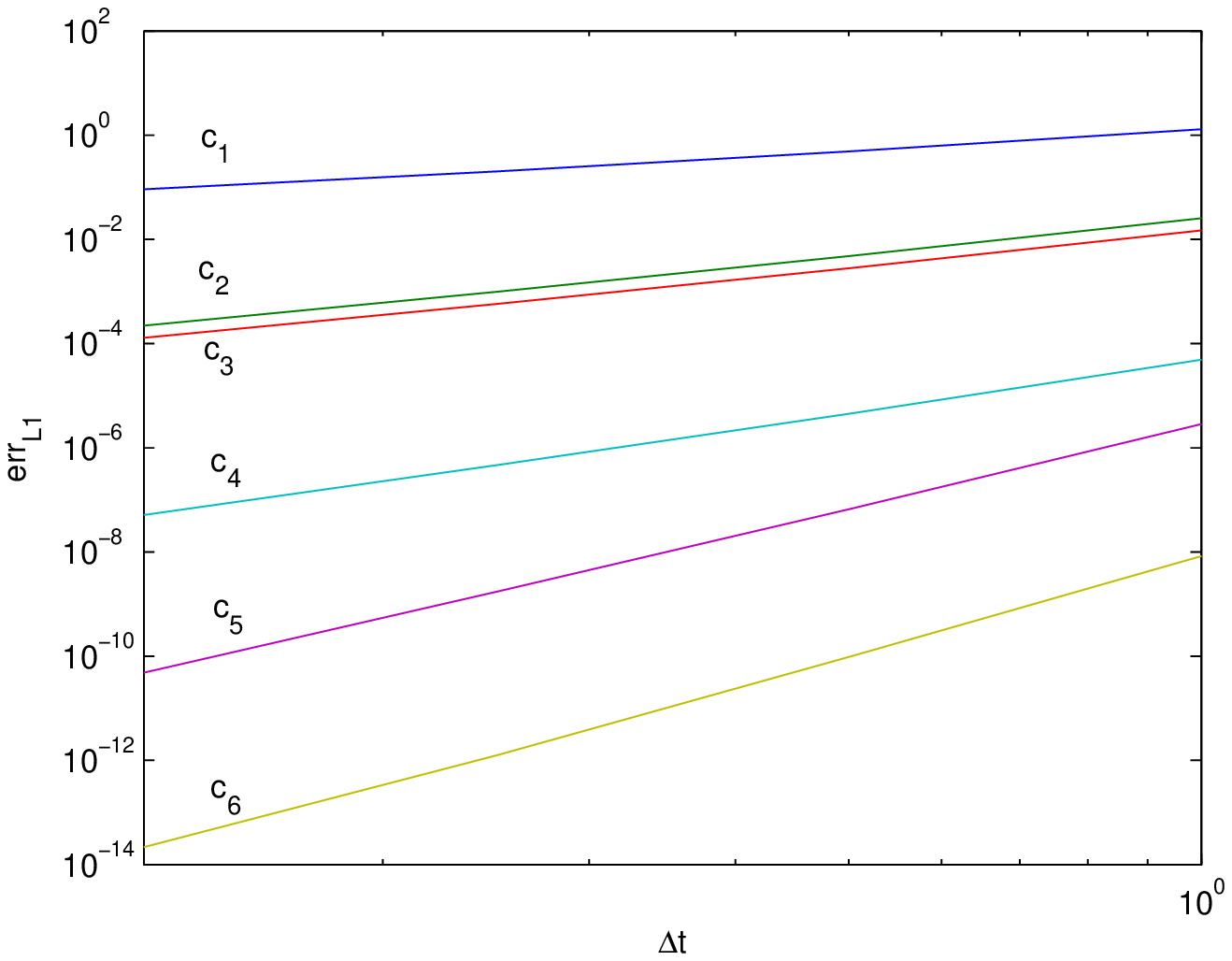} 
\end{center}
\caption{\label{first_1} Numerical errors of the one-side Splitting scheme with $A$ (upper figure),  the one-side Splitting scheme with $B$ (middle figure) and the two-side iterative schemes with $1, \ldots, 6$ iterative steps (lower figure).}
\end{figure}

The computational results are given in the Figure \ref{first_2}, we present the one-side and two-side iterative results.
\begin{figure}[ht]
\begin{center}  
\includegraphics[width=9.0cm,angle=-0]{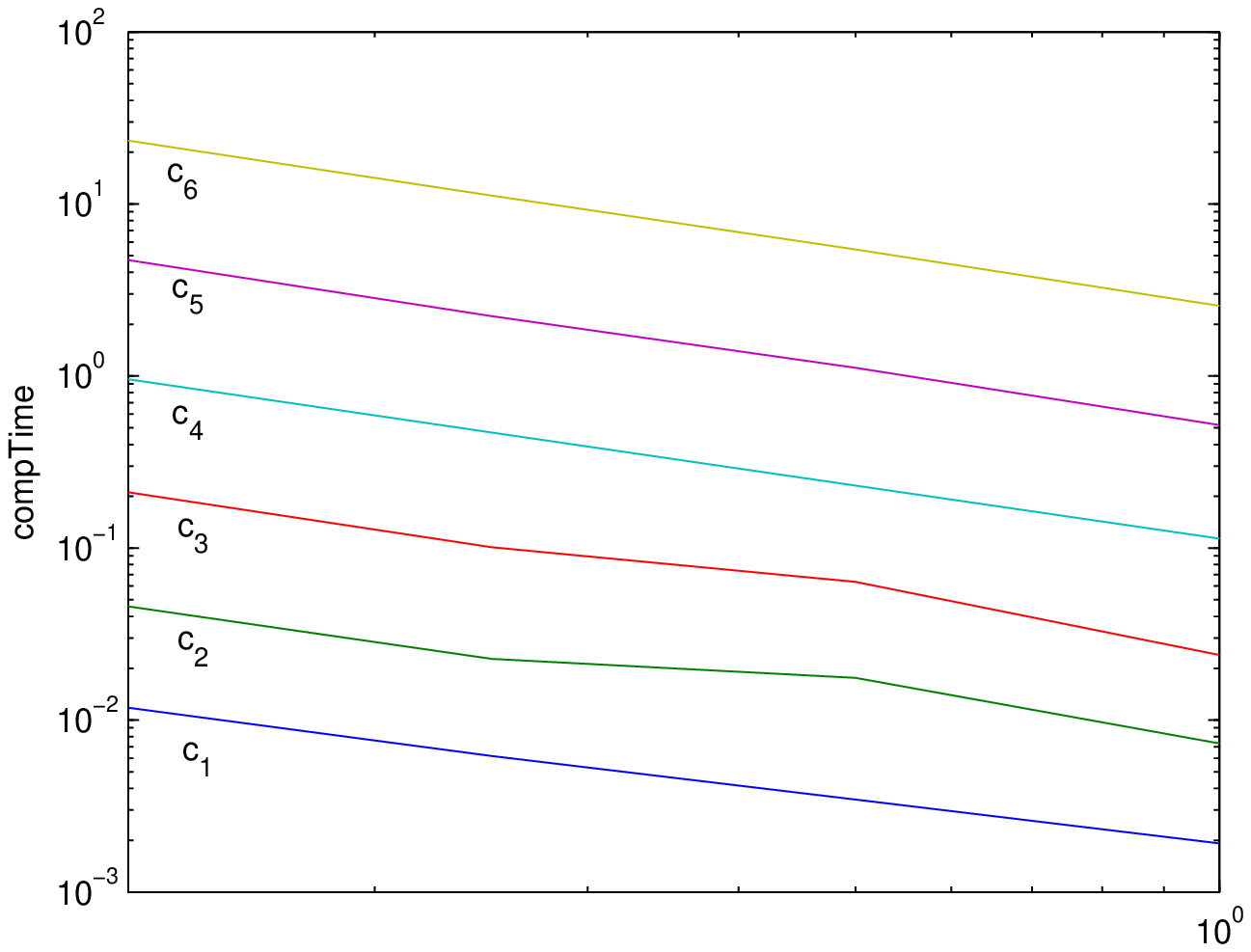} 
\includegraphics[width=9.0cm,angle=-0]{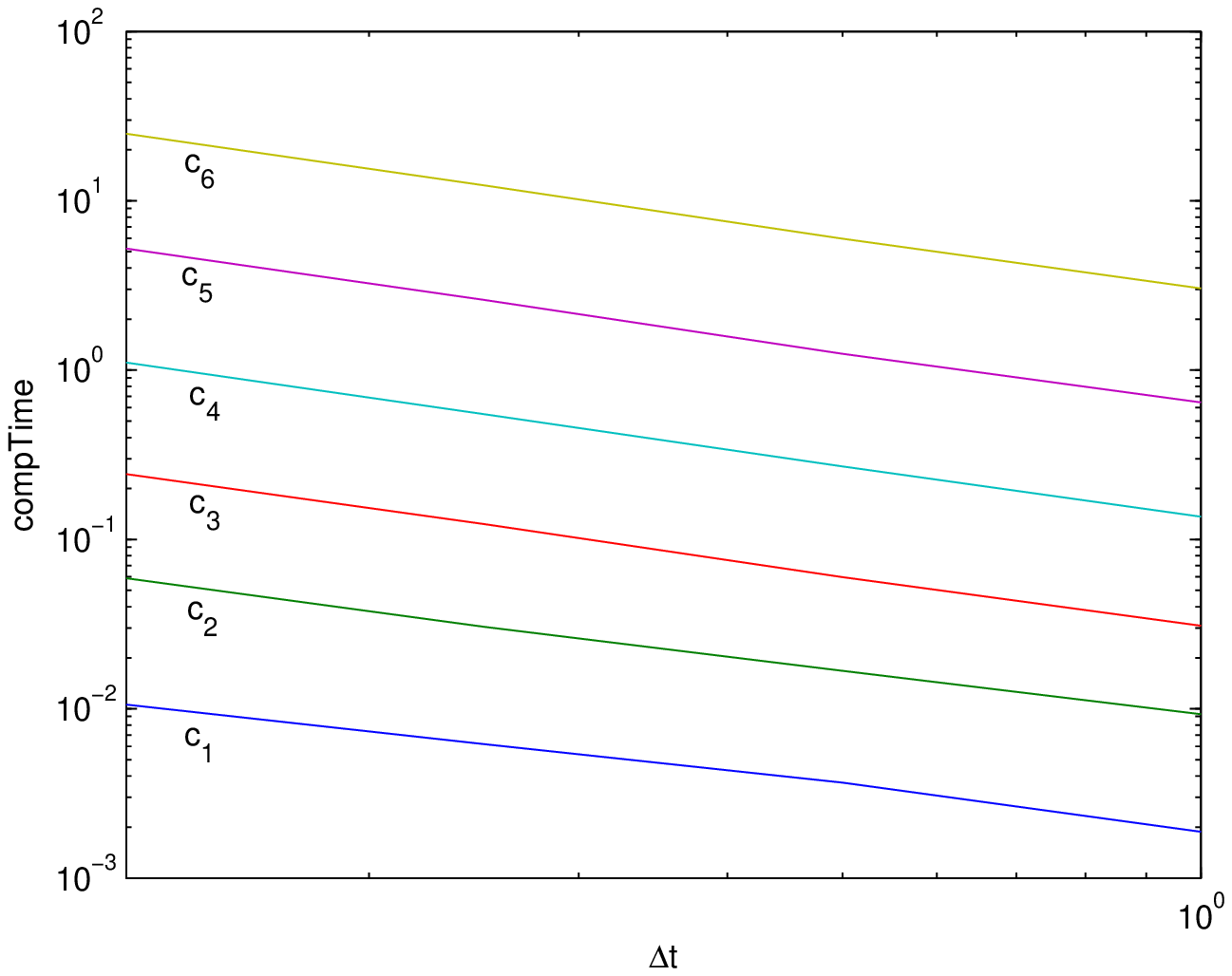} 
\includegraphics[width=9.0cm,angle=-0]{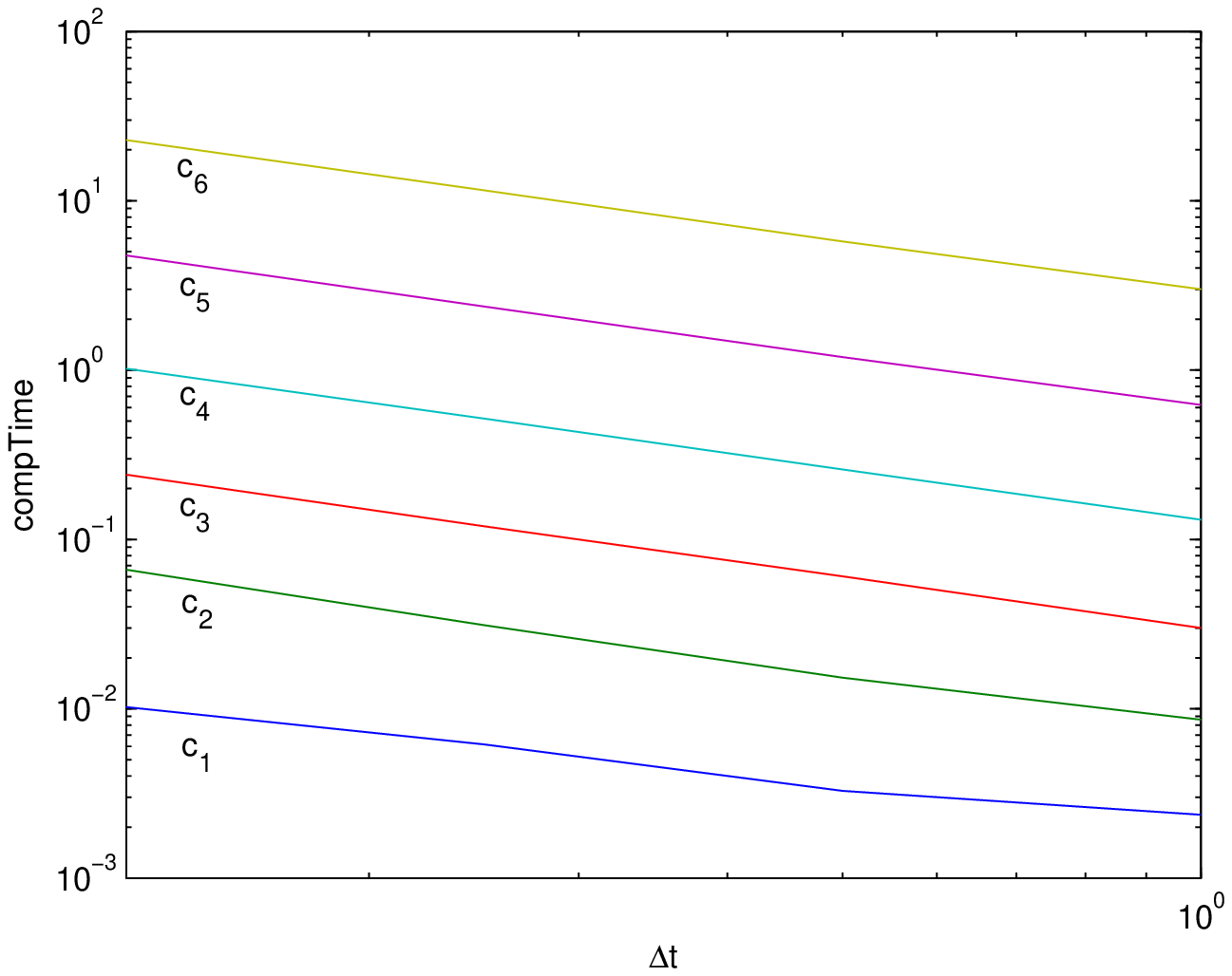} 
\end{center}
\caption{\label{first_2} The computational time of the one-side and two-side Splitting scheme: one-side splitting over $A$ (upper figure), one-side splitting over $B$ (middle figure) and  two-side splitting scheme alternating between $A$ and $B$ (lower figure) with $1, \ldots, 6$ iterative steps.}
\end{figure}

The Figure \ref{first_3} present the numerical errors between the exact and the
numerical solution for the optimized iterative schemes. Here we obtain results for one-side and two-side iterative schemes on operators $A$ and $B$. 
\begin{figure}[ht]
\begin{center}  
\includegraphics[width=9.0cm,angle=-0]{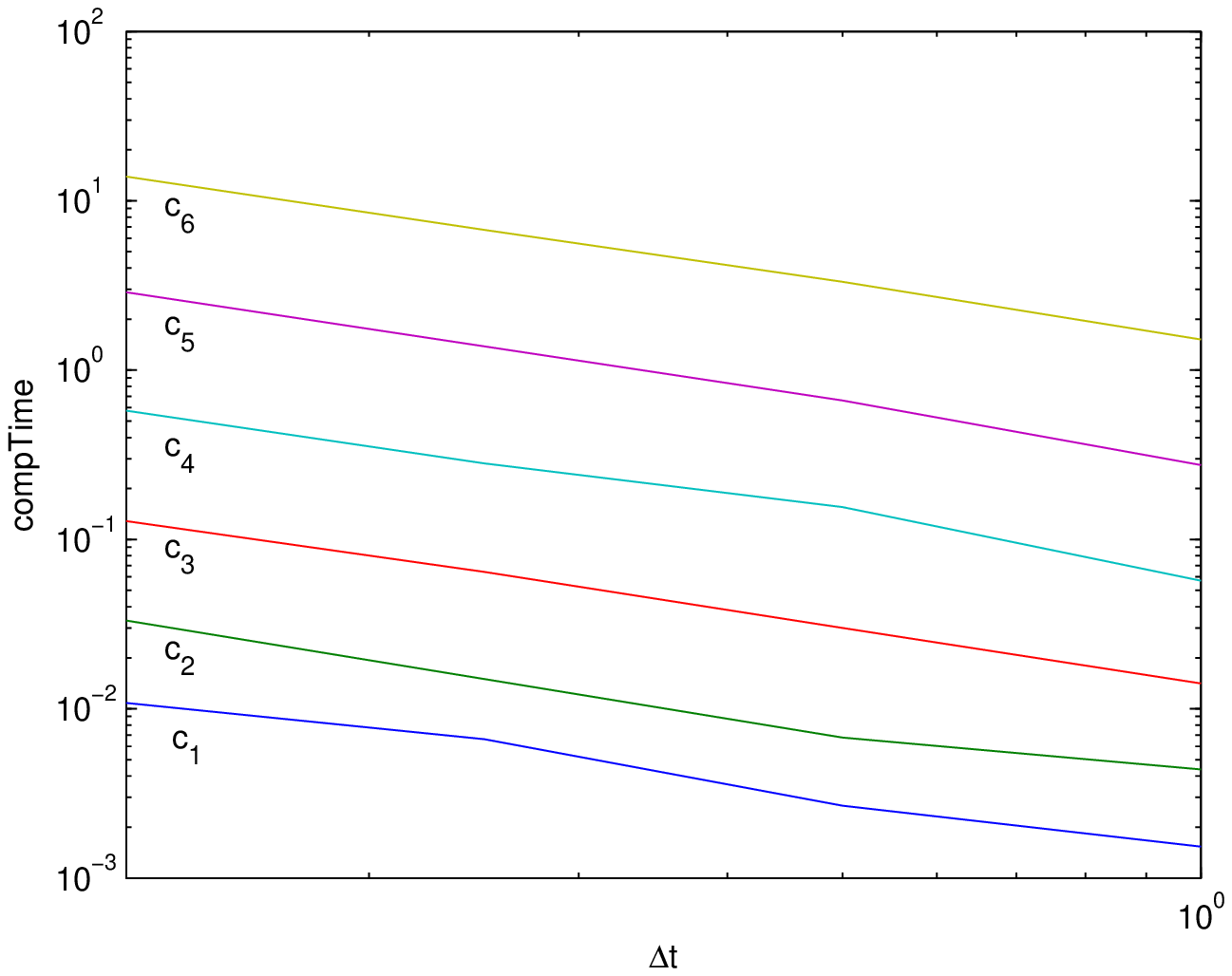} 
\includegraphics[width=9.0cm,angle=-0]{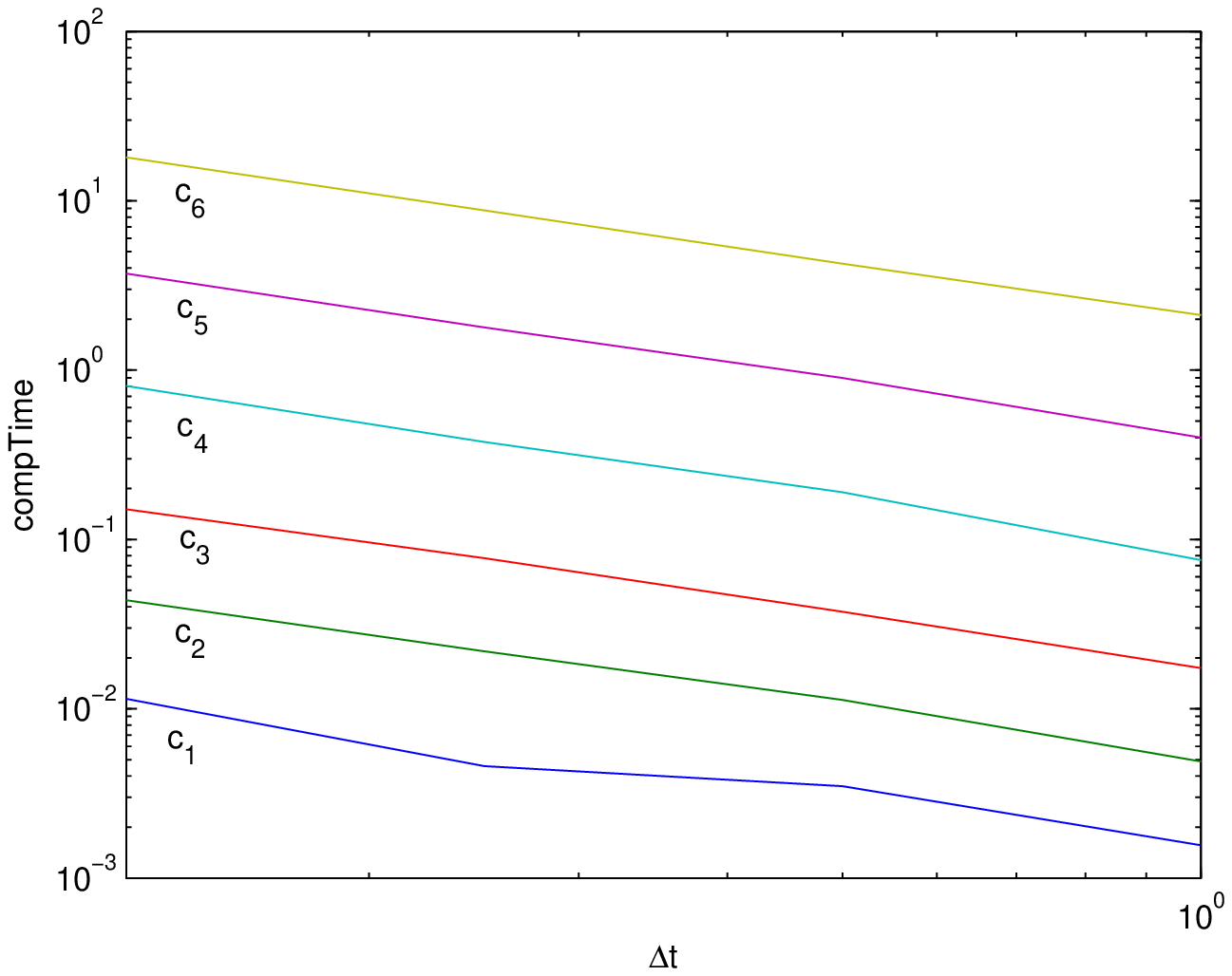} 
\includegraphics[width=9.0cm,angle=-0]{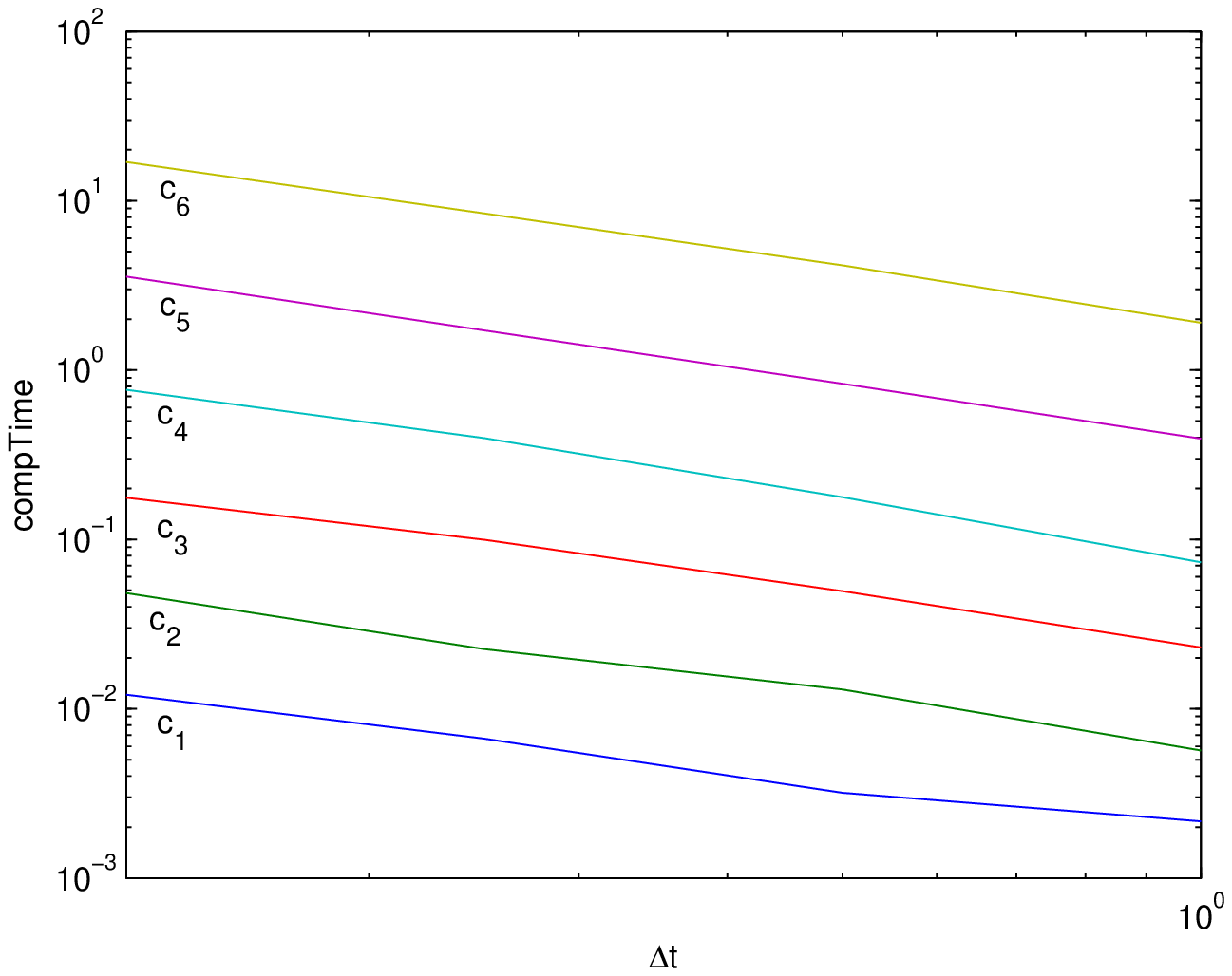} 
\end{center}
\caption{\label{first_3} Numerical errors of the one-side Splitting scheme with $A$ (upper figure),  the one-side Splitting scheme with $B$ (middle figure) and the two-side iterative schemes with $1, \ldots, 6$ iterative steps (lower figure).}
\end{figure}

\begin{remark}
For the computations, we see the benefit of the optimal iterative
schemes, which applied the two iterative steps of the two solutions
in one scheme, see Section \ref{splitt}.
The best results are given by the one-side iterative scheme with respect to
the operator $B$.
\end{remark}

\subsection{Second Example: Third order differential equations}

We deal with a simple third order differential equations:

\begin{eqnarray}
\label{equation1}
 && \frac{d^3c}{dt^3} - A c(t) = 0 , \; t \in [0,1], \\
&& c(0) = (1, \ldots, 1)^t \in \C^m, \\
&& c'(0) = \frac{1- \sqrt{2}}{3} A^{1/3} c(0),  \\
&& c''(0) = \frac{1}{3} A^{2/3} c(0)  ,
\end{eqnarray}
$A \in \C^m \times \C^m$, $c: \R^+ \rightarrow \C^m$ is sufficient smooth ($c \in C^3(\R^+)$) and we have $m = 10$.
 
The transformed first order differential equations are given as:
\begin{eqnarray}
\label{eq20}
 \partial_{t} c_1 -  A^{1/3} c_1  = 0 \\
 \partial_{t} c_2 -  A^{1/3} (- \frac{\sqrt{2}}{2} + i \frac{\sqrt{2}}{2} ) c_3  = 0 \\
 \partial_{t} c_3 -  A^{1/3} (- \frac{\sqrt{2}}{2} - i \frac{\sqrt{2}}{2} ) c_3  = 0 \\
\end{eqnarray}
where $c = \sum_{i=1}^3 d_i c_i(t)$ and $d_1, \ldots, d_3$ are given with respect to the initial conditions and are given as $d_1 = d_2 = d_3 = \frac{1}{3} c(0)$.

Further the operators for the splitting scheme for the three iterative splitting schemes are given as:
\begin{eqnarray}
\label{eq20}
 A_{1,1,re} = diag(A^{1/3}),  A_{1,2,re} outerdiag(A^{1/3}) , \\
 A_{2,1,re} = - \frac{\sqrt{2}}{2} diag(A^{1/3}),  A_{2,2,re} = - \frac{\sqrt{2}}{2} outerdiag(A^{1/3}) , \\
 A_{2,1,im} = \frac{\sqrt{2}}{2} diag(A^{1/3}),  A_{2,2,im} = \frac{\sqrt{2}}{2} outerdiag(A^{1/3}) , \\
 A_{3,1,re} = - \frac{\sqrt{2}}{2} diag(A^{1/3}),  A_{3,2,re} = - \frac{\sqrt{2}}{2} outerdiag(A^{1/3}) , \\
 A_{3,1,im} = - \frac{\sqrt{2}}{2} diag(A^{1/3}),  A_{3,2,im} = - \frac{\sqrt{2}}{2} outerdiag(A^{1/3}) , 
\end{eqnarray}

The matrix $A$ is given as
\begin{equation}
\label{num_9}
A = \left(
\begin{array}{c c c c c c c c c c}
- 0.01 & 0.01 & 0 &  \ldots\\
  0.01 & - 0.01 & 0 & \ldots \\
  0.01 & 0.01 & -0.02 & 0 & \ldots  \\
  0.01 & 0.01 & 0.01 & - 0.03 &  0 & \ldots  \\
  \vdots \\
  0.01 & 0.01 & 0.01 & 0.01 & 0.01 & 0.01 & 0.01 & 0.01 &  - 0.08 & 0  \\
  0.01 & 0.01 & 0.01 & 0.01 &  0.01 & 0.01 & 0.01 & 0.01 &  0 & -0.08  \\
\end{array} \right) \; ,
\end{equation}

Here, we deal with the following splitting schemes:

\begin{itemize}
\item $c_1$ is computed by a scalar iterative scheme.
\item $c_2, c_3$ are computed by a vectorial iterative scheme (because of
real and imaginary parts).
\end{itemize}

For $c_1$ we have:
\begin{eqnarray}
\label{gleich_oper_2} && \frac{\partial {\bf C}_1^i(t)}{\partial t} = {\cal A}_{11} {\bf C}_1^i(t) \; + \;  {\cal A}_{12} {\bf C}_1^{i-1}(t), \;
\mbox{with} \; \; {\bf C}_1^i(t^n) = {\bf C}_1^{i-1}({t^{n+1}}) \\
&& \mbox{and the starting values} \; {\bf C}_1^{0}(t^n) =  \frac{1}{3}{\bf C}(t^n) \;
\end{eqnarray}
where $ {\bf C}_1 = (c_{1,re} + i c_{1,im})^t$ and
\begin{eqnarray}
{\cal A}_{11} = \left(
\begin{array}{c c}
 A_{1,1,re}  & 0 \\
0 &  A_{1,1,re} 
\end{array} \right) ,
{\cal A}_{12} = \left(
\begin{array}{c c}
 A_{1,2,re}  & 0  \\
 0 &  A_{1,2,re} 
\end{array} \right) ,
\end{eqnarray}
for $i=1,2, \ldots, I$ and the solution is given as ${\bf C}_1^i(t^{n+1})$.

For $c_2$ we have:
\begin{eqnarray}
\label{gleich_oper_2} && \frac{\partial {\bf C}_2^i(t)}{\partial t} = {\cal A}_{21} {\bf C}_2^i(t) \; + \;  {\cal A}_{22} {\bf C}_2^{i-1}(t), \;
\mbox{with} \; \; {\bf C}_2^i(t^n) = {\bf C}_2^{i-1}({t^{n+1}}) \\
&& \mbox{and the starting values} \; {\bf C}_2^{0}(t^n) =  \frac{1}{3}{\bf C}(t^n) \;
\end{eqnarray}
where $ {\bf C}_2 = (c_{2,re} + i c_{2,im})^t$ and
\begin{eqnarray}
{\cal A}_{21} = \left(
\begin{array}{c c}
 A_{2,1,re}  & 0 \\
0 &  A_{2,1,re} 
\end{array} \right) ,
{\cal A}_{22} = \left(
\begin{array}{c c}
 A_{2,2,re}  & - ( A_{2,1,im} +  A_{2,2,im} )  \\
 ( A_{2,1,im} +  A_{2,2,im} ) &  A_{2,2,re} 
\end{array} \right) ,
\end{eqnarray}
for $i=1,2, \ldots, I$ and the solution is given as ${\bf C}_2^i(t^{n+1})$.

For $c_3$ we have:
\begin{eqnarray}
\label{gleich_oper_2} && \frac{\partial {\bf C}_3^i(t)}{\partial t} = {\cal A}_{31} {\bf C}_3^i(t) \; + \;  {\cal A}_{32} {\bf C}_3^{i-1}(t), \;
\mbox{with} \; \; {\bf C}_3^i(t^n) = {\bf C}_3^{i-1}({t^{n+1}}) \\
&& \mbox{and the starting values} \; {\bf C}_3^{0}(t^n) =  \frac{1}{3}{\bf C}(t^n) \;
\end{eqnarray}
where $ {\bf C}_3 = (c_{3,re} + i c_{3,im})^t$ and
\begin{eqnarray}
{\cal A}_{31} = \left(
\begin{array}{c c}
 A_{3,1,re}  & 0 \\
0 &  A_{3,1,re} 
\end{array} \right) ,
{\cal A}_{32} = \left(
\begin{array}{c c}
 A_{3,2,re}  & - ( A_{3,1,im} +  A_{3,2,im} )  \\
 ( A_{3,1,im} +  A_{3,2,im} ) &  A_{3,2,re} 
\end{array} \right) ,
\end{eqnarray}
for $i=1,2, \ldots, I$ and the solution is given as ${\bf C}_3^i(t^{n+1})$.

The solution is given as:

${\bf C}^i(t^{n+1}) = \sum_{j=1}^3 {\bf C}_j^i(t^{n+1})$.

The computational results for the optimized iterative schemes are given in the Figure \ref{first_4}, we present the one-side and two-side iterative results.
\begin{figure}[ht]
\begin{center}  
\includegraphics[width=9.0cm,angle=-0]{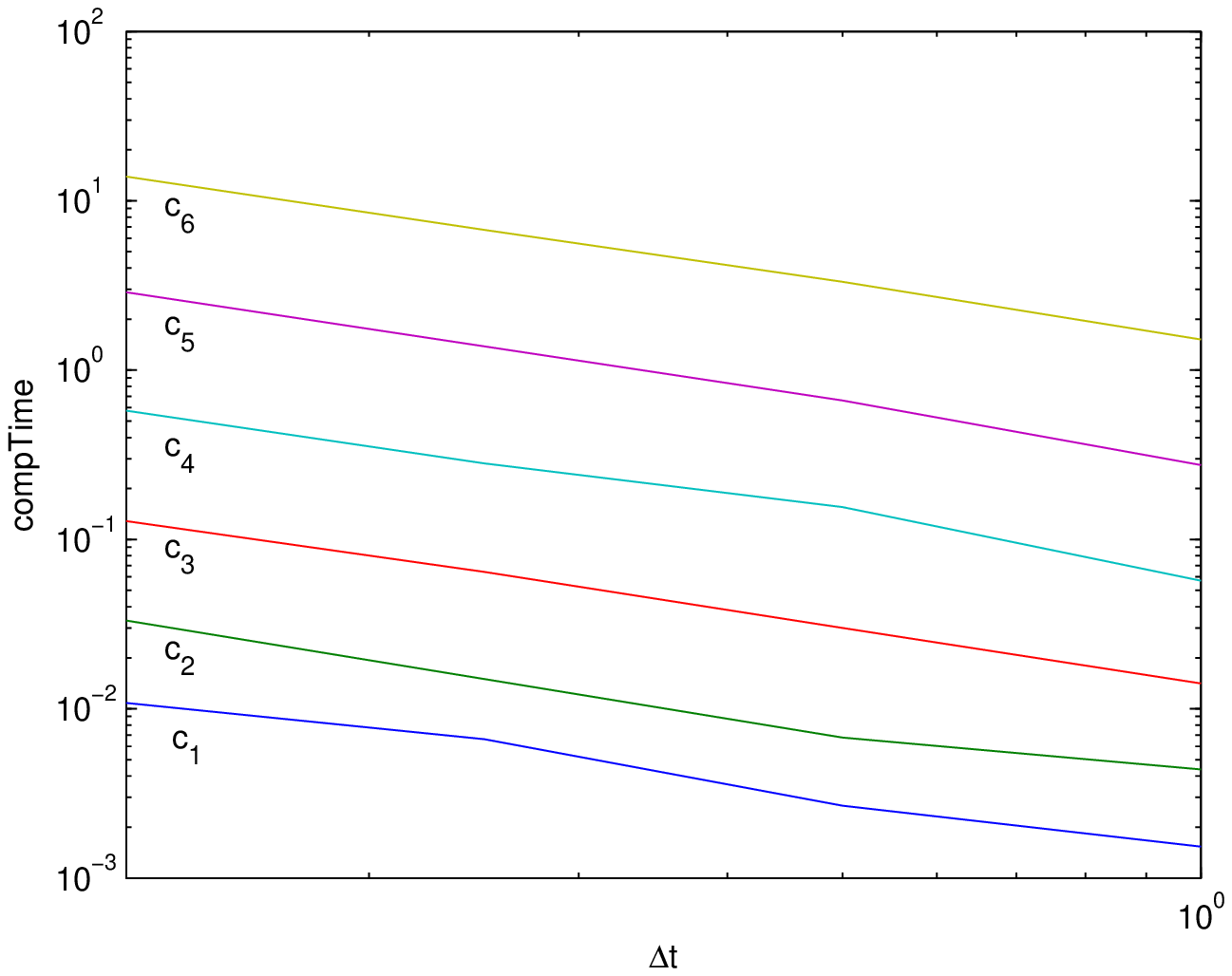} 
\includegraphics[width=9.0cm,angle=-0]{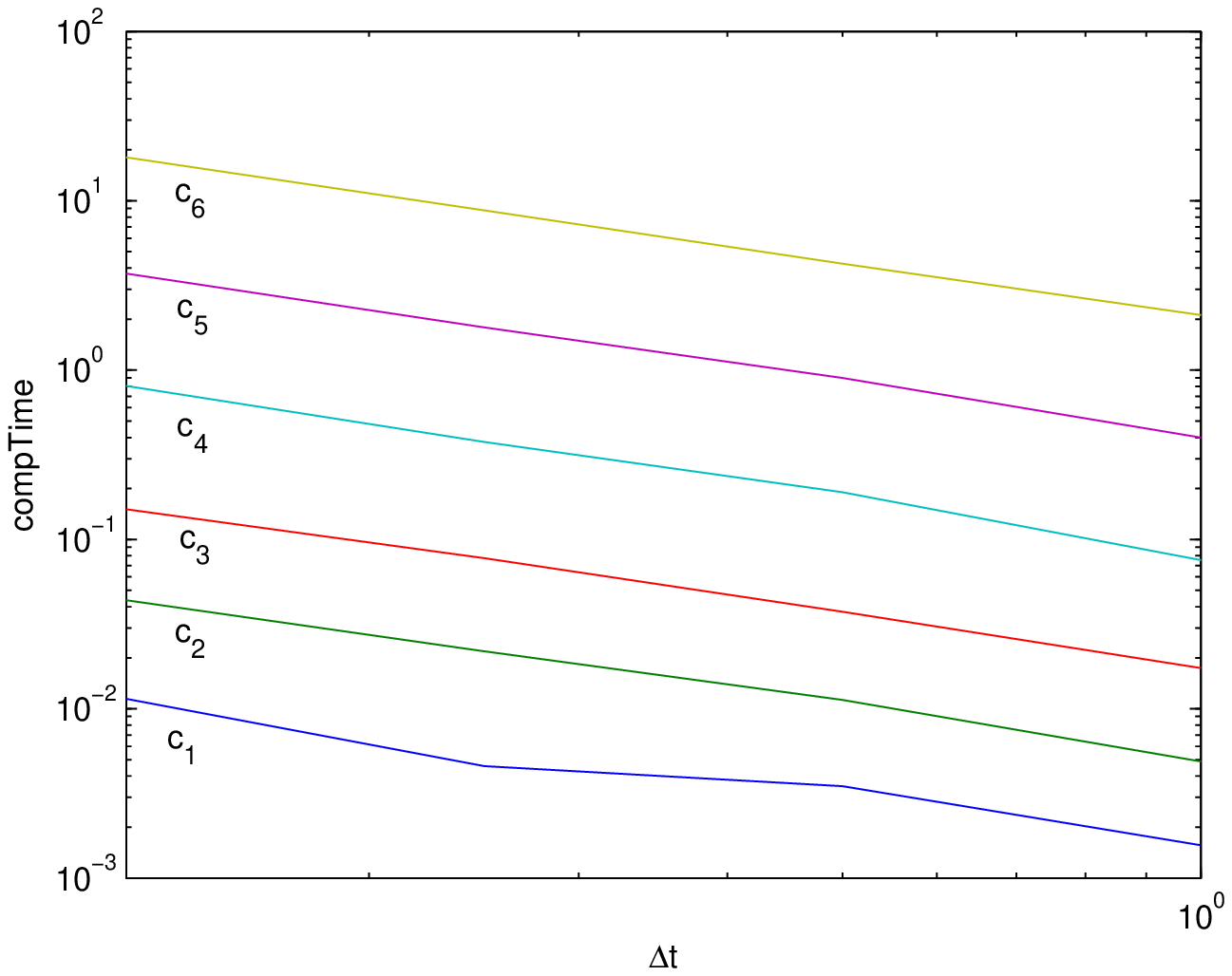} 
\includegraphics[width=9.0cm,angle=-0]{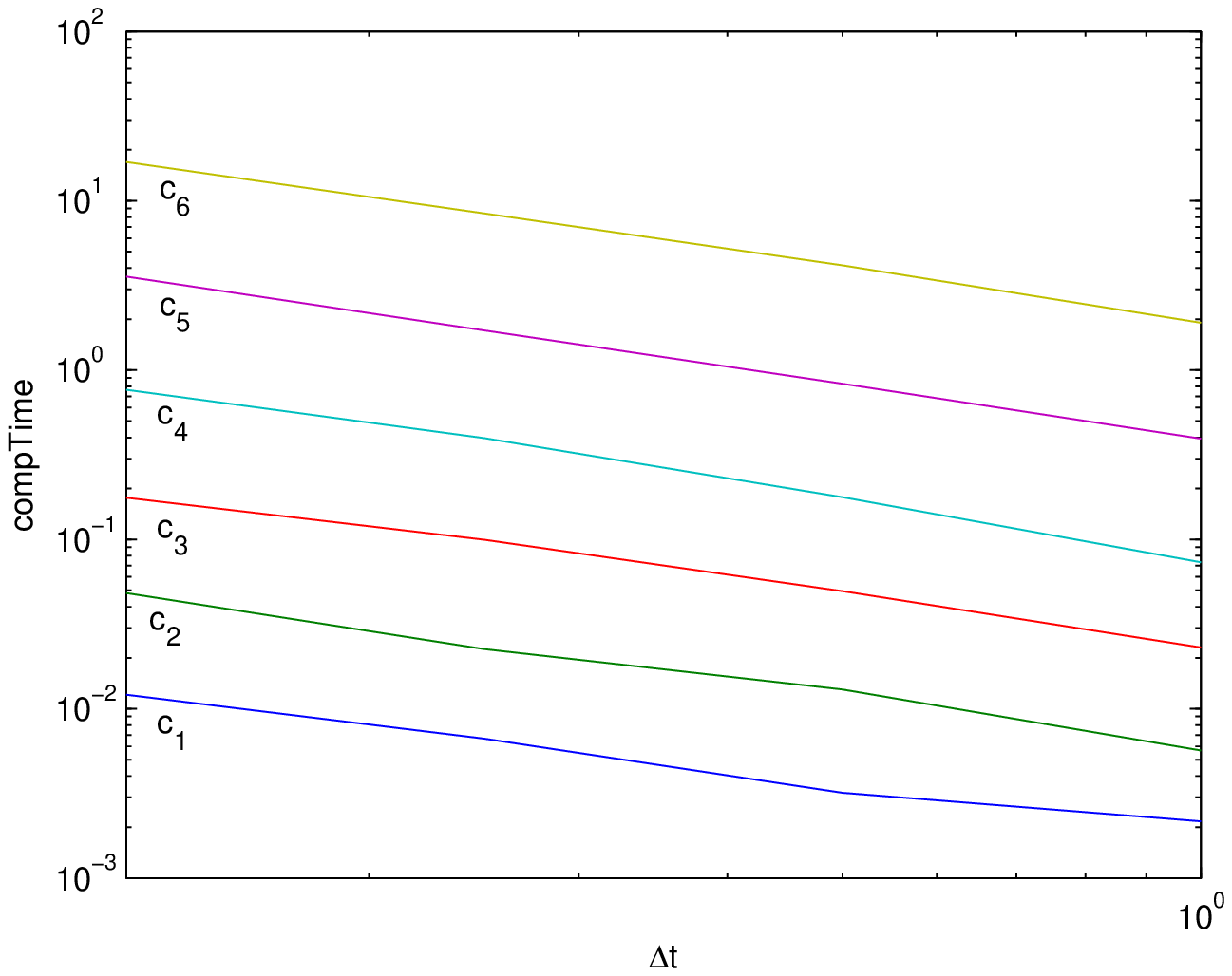} 
\end{center}
\caption{\label{first_4} The computational time of the one-side and two-side Splitting scheme: one-side splitting over $A$ (upper figure), one-side splitting over $B$ (middle figure) and  two-side splitting scheme alternating between $A$ and $B$ (lower figure) with $1, \ldots, 6$ iterative steps.}
\end{figure}

\begin{remark}
For the computations, we see the benefit of the optimal iterative
schemes. While we deal with real and imaginary parts, it is 
important to reduce the computational costs. 
We applied in one scheme the real and imaginary solution, see Section \ref{splitt}.
The best results are given by the one-side iterative scheme with respect to
the operator $B$.
\end{remark}

\section{Conclusions and Discussions }
\label{concl}

We present the coupled model for a transport model
for deposition species in a plasma environment.
We assume the flow field is computed by the plasma 
model and the transport of the deposition species 
with a transport-reaction model.

Based on the physical effects, we deal with higher
order differential equations (scattering parts, reaction parts, etc.).
We validate a novel splitting schemes, that embedded the real and 
imaginary parts of the solutions.
Standard iterative splitting schemes can be extended to such 
complex iterative splitting schemes.
First computations help to understand the 
important modeling of the plasma environment in 
a CVD reactor with scattering and higher order time-derivative parts.
In future, we work on a general theory of embedding the 
complex schemes to standard splitting schemes.

\bibliographystyle{plain}

\begin{thebibliography}{10}


\bibitem{blanes2009}  
S.~Blanes, F.~Casas, J.A.~Oteo and J.~Ros.
\newblock{\em The Magnus expansion and some of its applications.}
\newblock Physics Reports 470, 151-238, 2009.

\bibitem{braith2009}
N.St.J.~Braithwaite and R.N.~Franklin.
\newblock {\em Reflections on electrical probes.}
\newblock Plasma Resource Sci. Technol., 18, 014008, 2009.



\bibitem{cheng2011}
S.H.~Cheng, N.J.~Higham, C.S.~Kenney and A.J.~Laub.
\newblock{\em Approximating the Logarithm of a Matrix to Specified Accuracy},
\newblock SIAM Journal on Matrix Analysis and Applications 22 (4): 1112–1125, 2001.

\bibitem{denman1976}
E.D.~Denman and A.N.~Beavers.
\newblock{\em The matrix sign function and computations in systems}, 
\newblock Applied Mathematics and Computation 2 (1): 63–94, 1976.


\bibitem{EN00}
K.-J. Engel and R. Nagel,
\newblock {\em One-Parameter Semigroups for Linear Evolution Equations}.
\newblock Springer, New York, 2000.


\bibitem{farago04}
I.~Farago.
\newblock {\em Splitting methods for  abstract Cauchy problems.}
\newblock  Lect. Notes Comp.Sci. 3401, Springer Verlag, Berlin, 2005, pp. 35-45


\bibitem{fargei05}
I.~Farago, J.~Geiser.
\newblock {\em Iterative Operator-Splitting methods for Linear Problems.}
\newblock Preprint No. 1043 of the Weierstrass Institute for Applied Analysis and Stochastics, Berlin, Germany, June 2005.

\bibitem{gei01}
J.~Geiser.
\newblock {\em Numerical Simulation of a Model for
               Transport and Reaction of Radionuclides}.
\newblock Proceedings of the Large Scale Scientific Computations of Engineering           and Environmental Problems, Sozopol, Bulgaria, 2001.

\bibitem{gei_2009_5}
J.~Geiser.
\newblock{\em Decomposition Methods for Partial Differential Equations: Theory
and Applications in Multiphysics Problems.}
\newblock Numerical Analysis and Scientific Computing Series, CRC Press, Chapman \& Hall/CRC , edited by Magoules and Lai, 2009.

\bibitem{geiser_2011_1}
J.~Geiser.
\newblock{\em Computing Exponential for Iterative Splitting Methods.}
\newblock Journal of Applied Mathematics, special issue: Mathematical and Numerical Modeling of Flow and Transport (MNMFT), Hindawi Publishing Corp., New York, accepted, January 2011. 

\bibitem{glow04}
I.~Glowinski.
\newblock {\em The iterative Operator-Splitting methods}.
\newblock Preprint, University of Houston, 2004.

\bibitem{hun96} W.H.~Hundsdorfer.
\newblock {\em  Numerical Solution od Advection-Diffusion-Reaction Equations.}
\newblock Technical Report NM-N9603, CWI, 1996.


\bibitem{HunVer03} W.H.~Hundsdorfer, J. Verwer W.
 \newblock {\em Numerical solution of
time-dependent advection-diffusion-reaction equations},
\newblock {Springer,} Berlin, (2003).

\bibitem{gei_2011}
J.~Geiser.
\newblock{\em Iterative Splitting Methods for Differential Equations.}
\newblock Numerical Analysis and Scientific Computing Series, CRC Press, Chapman \& Hall/CRC , edited by Magoules and Lai, 2011.


\bibitem{kan03} J.\,Kanney, C.\,Miller and C.\, Kelley.
\newblock {\em Convergence of iterative split-operator approaches for
approximating nonlinear reactive transport problems.}
\newblock Advances in Water Resources, 26:247--261, 2003.

\bibitem{lapke2010}
M.~Lapke, Th.~Mussenbrock and R.P.~Brinkmann.  
\newblock{\em Modelling of volume- and surface wave based plasma resonance spectroscopy.}
\newblock Abstracts IEEE International Conference on Plasma Sciences, pp. 8-9, 2010.


\bibitem{luther2004} 
U.~Luther and Karla Rost.
\newblock{\em Matrix exponentials and Inversion of COnfluent Vandermonde Matrices.}
\newblock Electronic Transactions on Numerical Analysis, 18:91-100, 2004.


\bibitem{oberrath2011}  
J.~Oberrath, M.~Lapke, T.~Mussenbrock and R.P.~Brinkmann. 
\newblock {\em A Functional Analytical Description of Active Plasma Resonance Spectroscopy in Terms of Kinetic Theory}.
\newblock Proceeding of the 30th ICPIG, Belfast, August 28- September 2, 2011.



\bibitem{oteo1991} Z.\,Zlatev.
J.A.~Oteo and J.~Ros.
\newblock{\em The Magnus expansion for classical Hamiltonian systems.}
\newblock J. Phys. A: Math. Gen., 24 5751, 1991.

\bibitem{rhandi2002}
A.~Rhandi.
\newblock{\em Spectral Theory for Positive Semigroups and Applications.}
\newblock Quaderno Q. 1-2002, 51 pages, University of Lecce, Italy, 2002.


\bibitem{sene06}  T.K.~Senega and R.P.~Brinkmann. 
\newblock {\em A multi-component transport model for non-equilibrium
low-temperature low-pressure plasmas}.
\newblock J. Phys. D: Appl.Phys., 39, 1606--1618, 2006.

\bibitem{stra68}
G.~Strang.
\newblock {\em On the construction and comparision
of difference schemes}.
\newblock SIAM J. Numer. Anal., 5:506--517, 1968.

\bibitem{verwer98} J.,G.~Verwer and B.~Sportisse.
\newblock {\em A note on operator splitting in a stiff linear case.}
\newblock MAS-R9830, ISSN 1386-3703, 1998.

\bibitem{zla95} Z.\,Zlatev.
\newblock {\em Computer Treatment of Large Air Pollution Models.}
\newblock Kluwer Academic Publishers, 1995.

\end{thebibliography}

\end{document}